\documentclass[12pt]{article}
\usepackage{graphicx,amsmath,amssymb,amsthm} 

\numberwithin{equation}{section}

\theoremstyle{plain}
\newtheorem{thm}{Theorem}[section]

\newtheorem{cor}[thm]{Corollary}
\newtheorem{problem}{Problem}

\theoremstyle{definition}
\newtheorem{df}[thm]{Definition}
\newtheorem{nt}[thm]{Notation}
\newtheorem{rmk}[thm]{Remark}

\newcommand{\m}[1]{{\mathbf{\uppercase{#1}}}}
\newcommand{\converse}[1]{{\breve{#1}}}
\newcommand{\re}{{\rm Rel}}
\newcommand{\rel}{{\rm Rel}}
\newcommand{\mrel}{{\bf Rel}}
\newcommand{\fn}{{\rm Fn}}
\newcommand{\at}{{\rm At}}
\newcommand{\pr}{{\rm Pr}}
\newcommand{\id}{{\rm 1\hbox{\textquoteright}}}

\newcommand{\Z}{{\mathbb{Z}}}
\renewcommand{\;}{\mathbin{;}}

\title{Relation Algebras Compatible with $\mathbb{Z}_2$-sets}
\author{Jeremy F.~Alm, John W.~Snow}
\date{January 2025}

\begin{document}

\maketitle
\begin{abstract}
We provide a characterization of those relation algebras which are isomorphic to the algebras of compatible relations of some $\Z_2$-set. It follows that  this class is finitely axiomatizable in first-order logic in the language of relation algebras.\\
MSC 03G15. Keywords: relation algebra, representation, group action, axiomatizability. 
\end{abstract}

\section{Introduction}

In \cite{Jonsson1984} J\'onsson introduced a type of relation algebra representation related to group actions. Suppose that $G$ is a group of permutations on a set $A$, and let $\m A=\langle A, G \rangle$ be the associated $G$-set viewed as a unary algebra. A binary relation $R$ on $A$ is \emph{compatible} with $\m a$ if for every $\langle x,y \rangle \in R$ and for every $g \in G$ the pair $\langle g(x), g(y) \rangle$ is also in $R$.  The set of all binary relations on $A$ compatible with $\m a$ is closed on the relation algebra operations. J\'onsson calls a relation algebra arising in this way \emph{Galois closed} and proves that every complete concrete relation algebra whose atoms are functional is Galois closed.  In this article, we address the question of which abstract relation algebras have representations which are Galois closed. 
We call a relation algebra {\bf B}  \emph{group-action representable} (or say that $\m b$ is a \textsf{GARRA}) if for some group $G$ there is a $G$-set $\m A$ so that $\m B$ is isomorphic to the relation algebra of binary relations which are compatible with $\m a$.   After providing some examples of \textsf{GARRA}s, we give a characterization of all \textsf{GARRA}s whose associated group is $\Z_2$.  These turn out to be the collection of simple, pair-dense relation algebras in which every atom or its converse is functional. (In a concrete algebra, pair-density means every element below the identity contains a ``pair'' $\{( a,a ), ( b,b )\}$.) An immediate corollary to this result is that the collection of \textsf{GARRA}s whose associated group is $\Z_2$ is finitely  axiomatizable in first-order logic in the language of relation algebras.

\section{Background}

We assume that the reader is familiar with relation algebras and group actions.
For background information on relation algebras, we refer the reader to \cite{MadduxBook}.  For background information on group actions, we recommend any standard graduate text on algebra such as \cite{DummitAndFoote}.

\begin{df}
    Let $G$ be a group, let $A$ be a non-empty set, and let $G$ act on $A$, which we will denote by left multiplication. Let $\m A$ be the algebra with universe $A$ and the unary operations that arise from the action of $G$ on $A$. We say that a binary relation $R\subseteq A\times A$ is a \emph{compatible relation} on $\m A$ if for every $(x,y)\in R$ and every $g\in G$, $(gx,gy)\in R$. 
    Let $\rel(\m A)$ be the set of compatible relations on the $G$-set $\m A$. We also use $\rel(U)$ to represent the set of all binary relations on any set $U$.
\end{df}

That is, a compatible relation on $\m A$ is merely a subuniverse of $\m A^2$ in the language of universal algebra (for background information on universal algebra, which is not necessary to read further, see \cite{alvi}).  If $\m a$ is a $G$-set, then $\rel(\m a)$ is closed under the relation algebra operations in $A^2$. It is an easy exercise to that
the identity and universal relations are compatible and to
prove closure under intersection, converse, and composition.  That $\rel(\m a)$ is closed under unions follows from the fact that the operations of $\m a$ are unary.  Closure under complementation follows from the fact that the operations of $\m A$ are bijective and their inverses are also opearations of $\m A$.

\begin{df}
Suppose that $\m a$ is a $G$-set. We denote the relation algebra 
$$\langle \rel(\m A), \cup, \cap,\, ^{c}, \emptyset, A^2, \circ, \, ^{-1} , {\rm id}_A \rangle $$ as $\mrel(\m a).$
\end{df}

\begin{rmk}
Some researchers in relation algebra will denote the relation $R\circ S$ as $R|S$ and will use $R \circ S$ to represent $S|R$ (for example, see \cite{MadduxBook}). 
The results in this manuscript are focused on binary relations which are compatible with universal algebras. Familiar examples of such relations include congruence relations and tolerances. We have intentionally chosen to define composition with this order and with these symbols because doing so is consistent with traditional usage in universal algebra
(for example see \cite{alvi} and \cite{burrisSanka} for notation and order and \cite{gratzer} for order).
 Generally, we will use $R \circ S$ when  $R$ and $S$ are relations on a concrete set. We will use $R\;S$ when $R$ and $S$ are elements of an abstract relation algebra.  $R^c$ denotes the complement of a relation $R$.  $R^{-1}$ denotes the converse of $R$, and ${\rm id}_A$ denotes the identity relation on $A$.
\end{rmk}

\begin{df}
    Given a relation algebra $\m B$, we say that $\m B$ has a \emph{group-action representation} or is \emph{group-action representable} or is a \emph{\textsf{GARRA}}  if there exists a group $G$ and a $G$-set $\m a$ such that $\m B$ is isomorphic to $\mrel(\m a)$.  If there is such a $G$ and $\m A$, we say that $\m B$ is \emph{$G$-action representable}.
\end{df}

 In \cite{Jonsson1984}, J\'onsson calls concrete relation algebras of the form $\mrel(\m A)$ for some $G$-set $\m a$ \emph{Galois-closed} since these are the closed sets in the Galois connection between binary relations and permutations on the universe of $\m a$. J\'onsson addresses Galois-closed relation algebras whose atoms are functional. Below, we will be concerned with certain relation algebras in which every atom or its converse is functional.

The representation result we prove here concerns pair-dense relation algebras, introduced by Maddux in \cite{Maddux1991}. 

\begin{df}
Let $\m B$ be a relation algebra, and let $x$ be a nonzero element of $\m B$. We say that $x$ is a \emph{point} if $x \; 1 \; x \leq 1\hbox{\textquoteright}$.
We say that $x$ is a pair if $x \; 0\hbox{\textquoteright} \; x \; 0\hbox{\textquoteright} \; x \leq 1\hbox{\textquoteright}$.  A \emph{twin} is a pair that does not contain a point. $\m B$ is \emph{pair-dense} if every non-zero element below $1\hbox{\textquoteright}$ contains (is greater than or equal to) a pair.  
\end{df} 

\begin{rmk}
Here 
$1\hbox{\textquoteright}$ denotes the identity element, and $0\hbox{\textquoteright}$ is the diversity element (the complement of the identity).
\end{rmk}

The definition of point is intended to describe a one-element subset of the identity relation. The definition of twin is intended to identify a two-element subset of the identity.
In a subalgebra $\m B$ of the algebra $\rel(U)$ of binary relations on a set $U$, points are relations of the form $\{( a,a ) \}$, pairs are relations of the form $\{( a,a ), ( b,b )\}$, and twins are relations of the form $\{( a,a ), ( b,b )\}$ for which neither $\{( a,a )\}$ nor $\{( b,b )\}$ is a point in $\m B$. 

\begin{nt}
We will use $\at(\m B)$ to denote the set of atoms of a relation algebra $\m B$. $\pr (\m b)$ will denote the set of pairs in $\m b$. $\fn(\m b)$ is the set of functional elements in $\m b$ (elements $x$ which satisfy $x^{-1}\;x \leq 1\hbox{\textquoteright}$).
\end{nt}

\section{The Structure of Pair-Dense Algebras} \label{structure}

Here we briefly summarize the structure of pair-dense relation algebras derived by Maddux in \cite{Maddux1991}.
Suppose that $\m R$ is a simple, complete, pair-dense relation algebra.
Since $\m R$ is simple and pair-dense, $\m R$ is completely representable by Theorem 51 in \cite{Maddux1991} and is atomic by Theorem 48 in \cite{Maddux1991}.  Moreover, the discussion on pages 86 and 87 of \cite{Maddux1991} describe the representation exactly in the case when $\m R$ is complete.  There is a set $U$ and a relation algebra $\m B \subseteq \re (U)$ with $\m R \cong \m B$.   
Let $P$ be the set of all $\{a\} \subseteq U$ for which $\{( a,a )\}$ is a point in $\m B$. Let $T$ be the set of all two-element subsets $\{a,b\}$ of $U$ for which $\{( a,a ), ( b,b )\}$ is a pair which does not contain a point. 
$P \cup T$ is a partition of $U$ into one- and two-element subsets. The discussion on pages 86 and 87 of \cite{Maddux1991} describes the atoms of $\m B$.
There is an equivalence relation $\sim$ on $T$ so that the atoms of $\m B$ are specified in this way:
\begin{enumerate}
\item If $\{a\} \in P$, then $\{( a,a ) \}$ is an atom.
\item If $\{a,b\} \in T$ then $\{( a,a ), ( b,b )\}$ and $\{ ( a,b ), ( b,a ) \}$ are atoms.
\item If $\{a\},\{b\} \in P$, then $\{( a,b )\}$ and $\{ ( b,a )\}$ are atoms.
\item If $\{a,b\} \in T$ and $\{c\} \in P$, then $\{( a,c), ( b,c ) \}$ and its converse are atoms.
\item If $\{a,b \}, \{c,d\} \in T$ and $\{a,b\} \sim \{c,d\}$, then $\{ ( a,c), ( b,d )\}$ and $\{( a,d), ( b,c )\}$ and their converses are atoms.
\item If $\{a,b \}, \{c,d\} \in T$ and $\{a,b\} \not \sim \{c,d\}$, then $\{ ( a,c), ( b,d ), ( a,d), ( b,c )\}$ and its converse are atoms.
\end{enumerate}

\section{Examples}

For our first example, we start with a group action and see which relation algebra we get. Consider $\mathbb{Z}_3$ acting on itself.  We can calculate the atoms of the algebra of compatible relations as follows. 

Consider the pair $(0,0)$. By letting each element of $\mathbb{Z}_3$ act on $(0,0)$, we get the identity $\mathrm{Id} = \{(0,0), (1,1), (2,2) \}$. 
Similarly, by starting with $(0,1)$ and applying the action, we get $R = \{(0,1), (1,2), (2,0)\}$. 
Finally, by starting with $(1,0)$ and applying the action, we get $R^{-1} = \{(1,0), (2,1), (0,2)\}$.  This exhausts all nine pairs, and we have relation algebra $2_3$ (see \cite{MadduxBook} for the numbering system), with atoms $1\hbox{\textquoteright}$, $r$, and $\breve{r}$, with sole forbidden cycle $rrr$.  (A cycle of $abc$ of diversity atoms is called \emph{mandatory} if $a\;b\geq c$ and \emph{forbidden} if $a\;b\cdot c = 0$.)

For our second example, we will start with the abstract relation algebra and construct a group action. Consider relation algebra $5_7$, with atoms $1\hbox{\textquoteright}$, $a$, and $b$, all symmetric, with forbidden cycles $aaa$ and $bbb$. It is well known (folklore) that this algebra is representable over a 5-point set only \cite{MR1334290}. Up to isomorphism, its unique representation over $U=\{0,1,2,3,4\}$ has atoms as follows:
\begin{itemize}
    \item $\{(0,0),(1,1),(2,2),(3,3),(4,4)\}$
    \item $\{(0,1),(1,2),(2,3),(3,4),(4,0),(1,0),(2,1),(3,2),(4,3),(0,4)\}$
    \item $\{(0,2),(1,3),(2,4),(3,0),(4,1), (2,0), (3,1),(4,2),(0,3),(1,4)\}$
\end{itemize}

Now we define an action of the dihedral group $D_5 = \langle r,s\rangle$ on $U=\{0,1,2,3,4\}$, where $r = (0\ 1\ 2\ 3\ 4)$ and $s = (1\ 4)(2\ 3)$.  The interested reader is invited to check that applying every element of $D_5$ via this action to $(0,0)$, $(0,1)$, and $(0,2)$, respectively, generates the itemized sets above.

\section{Pair-Dense Relation Algebras and Group Actions}
\begin{thm} \label{atomsAlmostFunctions}
Suppose that $\m R$ is the algebra of binary relations compatible with a $G$-set $\m A$ where $G$ is a cyclic group of prime order $p$. If $r$ is an atom of $\m R$, then either $r$ or $\converse r$ is a function.
\end{thm}

\begin{proof}
Let $G$ be a group of prime order $p$ written in multiplicative notation, and let $g$ be a generator for $G$.  Let $\m A$ be a $G$-set, and let $\m R$ be the algebra of binary relations compatible with $\m A$.  Suppose that $r$ is an atom in $\m R$. There is an ordered pair of elements $( x,y )$ in $\m A$ so that $r=\{ ( g^n x, g^n y ) : n=0, 1, \ldots, p-1\}$.  If $gy=y$, then $r$ is the graph of a constant function (constantly $y$). If $gx=x$, then $\converse r$ is the graph of a constant function (constantly $x$).  If $gx\neq x$ and $gy \neq y$, then $x, gx, g^2x, \ldots, g^{p-1}x$ are $p$ distinct elements of $\m A$, and $y, gy, g^2y, \ldots, g^{p-1}y$ are $p$ distinct elements of $\m A$. Since no first or second coordinate in $r$ is repeated $r$ is a function (and so is $\converse r$).
\end{proof}

This condition that each atom or its converse is a function is essential to our main result below. However, it is interesting in and of itself as it implies representability for atomic relation algebras.  We first need a theorem due to Maddux:

\begin{thm} \label{Maddux1978} \upshape{(See Theorem 7 of \cite{Maddux1978} or Theorem G  of \cite{Maddux1991})}  If $\m R$ is a relation algebra and $\sum\{\converse x   	\; y : x,y \in \fn (\m R)\} = 1$ then $\m R$ is representable.
\end{thm}

\begin{thm}
If $\m R$ is an atomic relation algebra in which every atom or its converse is functional, then $\m R$ is representable.
\end{thm}

\begin{proof}
Suppose $\m R$ is an atomic relation algebra in which every atom or its converse is functional. We will use Theorem \ref{Maddux1978} to prove that $\m R$ is representable. Let $F$ be the set of atoms of $\m R$ which are functions, and let $N$ be the set of atoms of $\m R$ whose converses are functions.  
Also let $X=\sum\{\converse x \; y : x,y \in \fn \m R\} $. It is sufficient to prove that $1 \leq X$.
Note that if $x \in F$ then $x=\converse{\id} \; x \leq X$.
If $x \in N$ then $\converse x \in F$ and $x=\converse {\converse x} \; \id \leq X$.  Since $\m R$ is atomic, 
$$1=\sum_{x \in \at (\m R)} x=\left(\sum_{x \in F} x \right) + \left(\sum_{x \in N} x \right) \leq X+X=X.$$
By Theorem \ref{Maddux1978}, we can now conclude $\m R$ is representable.
\end{proof}

We now move on to our main theorem.

\begin{thm}
A  relation algebra $\m R$ is isomorphic to the algebra of binary relations compatible with a $G$-set where $G \cong \Z_2$ if and only if $\m R$ is simple, complete, and pair-dense, each atom of and $\m R$ or its converse is a function.
\end{thm}

\begin{proof}
For the entirety of this proof suppose $G=\{1_G,g\}$ is a two-element group with multiplicative notation. 
Let $\m A$ be a $G$-set, and let $\m R$ be the algebra of binary relations compatible with $\m A$. 
By Theorem \ref{atomsAlmostFunctions} we know that each atom in $\m R$ or its converse is a function. In $\m R$, $1=A \times A$, so $\m R$ is simple. 
The relations in $\m R$ form an algebraic lattice in the universal algebra sense, so $\m R$ is complete. 
We need only show that $\m R$ is pair-dense.  Suppose that $( a,a ) \in \id$ in $\m R$.  Let $r$ be the universe of the subalgebra of $\m A^2$ generated by $( a,a )$. Then $r=\{( a,a), ( ga, ga ) \}$ is a pair and $( a,a ) \in r$. It follows that $\id = \sum \pr (\m R)$, so $\m R$ is pair-dense.  

Now suppose that $\m R$ is a simple, complete, pair-dense relation algebra and that each atom of $\m R$ or its converse is a function.  We must prove that $\m R$ is isomorphic to the algebra of binary relations compatible with a $G$-set.  Let $U$, $\m B$, $T$, $P$, and $\sim$ be as in the discussion of the structure of pair-dense relation algebras in section \ref{structure}.

Consider an atom $r=\{ ( a,c), ( b,d ), ( a,d), ( b,c )\}$ of type 6.  It must be that $c \neq d$ since the pairs in $T$ are distinct. Since $r$ contains $( a,c )$ and $( a,d )$, $r$ is not a function. Moreover, it must be that $a \neq b$  since the pairs in $T$ are distinct. Since $\converse r$ contains $( c,a )$ and $( c,b )$, $\converse r$ is not a function. This contradicts our assumption that every atom of $\m R$ or its converse is a function. Therefore, there are no atoms of type 6 in $\m B$. This implies that $\sim$ is the total relation on $T$ and that the only atoms of $\m B$ are of types 1--5.  

We now prove that $\m B$ is the algebra of binary relations compatible with a $G$-set on $U$. To do so, we define an action of $G=\{1_G,g\}$ on $U$. The element $1_G$ must act as the identity. If $\{a\} \in P$, define $ga=a$. If $\{a,b\} \in T$ define $ga=b$ and $gb=a$.  This defines an action of $G$ on $U$. Denote the resulting $G$-set as $\m U$. We must show that $\m B=\rel (\m U)$.  A quick check will show that each atom of types 1-5 above is closed under the action of $G$. This implies that the atoms of $\m B$ are in $ \rel (\m U)$. Since $\m B$ is atomic, $\m B \subseteq \rel (\m U)$.  

Now suppose that $r$ is an atom of $\rel (\m U)$. This implies that there exist $a,b \in U$ so that $r=\{( a,b ), ( ga, gb )\}$.  We proceed by cases on which of $\{a\}$ and $\{b\}$ may be in $P$.  Suppose that $\{a\},\{b\} \in P$. Then $ga=a$ and $gb=b$, so $r=\{( a,b )\}$ with $\{a\},\{b\} \in P$. In this case, $r$ is an atom of $\m B$ of type 3.  Now suppose that $\{a\} \in P$ and $\{b\} \not \in P$. Since $\{a\} \in P$, $ga=a$. Since $\{b\} \not \in P$, there is a $d \in U$ with $\{b,d\}\in T$. This implies $gb=d$ and $gd=b$. In this case, $r=\{( a,b ), ( a, d )\}$ with $\{a\} \in P$ and $\{b,d\} \in T$, so $r$ is an atom of $\m B$ of type 4 (it is the converse of the kind explicitly displayed in type 4). Finally, suppose that $\{a\} \not \in P$ and $\{b\} \not \in P$. There are two subcases here. Either $\{a,b\} \in T$ or not. If $\{a,b\} \in T$, then $ga=b$ and $gb=a$. This implies $r=\{( a,b ), ( b, a )\}$.  In this case, $r$ is an atom of $\m B$ of type 2 (the second kind in type 2).  Finally, if neither of $\{a\}$ or $\{b\}$ is in $P$ and if $\{a,b\} \not \in T$, then there must be $c,d\in U$ so that $\{a,c\} \in T$ and $\{b,d\} \in T$. This implies that $ga=c$, $gc=a$, $gb=d$, and $gd=b$. Therefore, $r=\{( a,b ), ( c, d )\}$.  In this case, $r$ is an atom of $\m B$ of type 5. 

We have proven that every atom of $\rel (\m U)$ is also an atom of $\m B$. Since $\rel (\m U)$ is atomic, $\rel (\m U) \subseteq \m B$ (note that $\m B$ is complete by Maddux's characterization). Since $\m R \cong \m B$, we now that that $\m R$ is isomorphic to the algebra $\rel (\m U)$ of binary relations compatible with the $G$-set $\m U$.

\end{proof}

\begin{cor}
    The collection $\mathbb{Z}_2$-$\textsf{GARRA}$ of \textsf{RA}s that are representable as algebras of compatible relations on a $\mathbb{Z}_2$-set is finitely axiomatizable  in first-order logic in the language of relation algebras. 
\end{cor}

\begin{proof}
    We give three axioms that imply, in turn, simplicity, pair-density, and the property that every atom or its converse is functional. (The axioms for relation algebras are assumed but not restated here.) 

    \begin{enumerate}
        \item $\forall x [(x>0) \rightarrow (1\;x\;1=1)]$
        \item $\forall x [([x>0] \wedge [x\leq 1\hbox{\textquoteright}]) \rightarrow \exists y ([0<y]\wedge [y\leq x] \wedge [y\;0\hbox{\textquoteright}\;y\;0\hbox{\textquoteright}\;y\leq 1\hbox{\textquoteright}])]$
        \item $\forall x  ( \forall y [(y<x) \rightarrow (y=0)] \rightarrow ([\breve{x}\;x\leq 1\hbox{\textquoteright}] \vee [x\;\breve{x}\leq 1\hbox{\textquoteright}])]$
    \end{enumerate}
\end{proof}

This is a $\Pi_2$ axiomatization; since neither $\textsf{GARRA}$ nor its complement $\textsf{RA}\setminus\textsf{GARRA}$ is closed under substructures, $\Pi_2$ cannot be replaced by $\Sigma_1$ nor $\Pi_1$. 

\section{Open Problems}

The notion of ``group-action representation'' seems to suggest many questions for further study. We state a few here. 

\begin{problem}
For a fixed odd prime $p$, is there a first-order characterization for relation algebras compatible with $\mathbb{Z}_p$-sets?
\end{problem}

Note that ``pair'' is natural to define in the language of relation algebra because relation algebra equations capture sentences in first order logic with no more than three variables \cite{TarskiGivant1987}. It is possible to say with three variables that a set contains two or fewer elements, but it is not possible to say that a set contains three or fewer elements.

\begin{problem}
    What is the relationship between group-representable relation algebras and group-action-representable relation algebras? Is $\textsf{GRA} $ properly contained in $ \textsf{GARRA}$? 
\end{problem}

\begin{problem}
    Which  relation algebras that are given a number in Maddux's numbering system have group-action representations? We have given two examples here. 
\end{problem}

\begin{problem}
    In \cite{Almetal2024}, the cyclic group spectrum, i.e., 
    
    \[
    \{n \in \mathbb{Z}^+ : \mathbf{A} \text{ is representable over } \mathbb{Z}/n\mathbb{Z} \}
    \]
    
    was determined for some small relation algebras. Can the ``group-action spectrum'' be determined for these algebras?  It would be a (possibly empty) subset of the ordinary spectrum. 
\end{problem}

\begin{problem}
If a relation algebra $\m b$ has a $\mathbb{Z}_2$-set representation, then $\m b$ is pair-dense, so every representation of $\m b$ has the structure described in \cite{Maddux1991}. This implies every representation of $\m b$ is a $\mathbb{Z}_2$-set representation.  Call a relation algebra $\m b$ \emph{strongly group-action representable} if every representation of $\m b$ is a group-action representation. Which relation algebras are strongly group-action representable?
\end{problem}

 \bibliographystyle{plain}
 \bibliography{references}
\end{document}